\documentclass{amsart}
\usepackage[T1]{fontenc}
\usepackage[english]{babel}
\usepackage{amsmath,amsthm,amssymb, latexsym,geometry,stackrel,enumerate}
\usepackage[all]{xy}
\usepackage{hyperref}

\usepackage[dvips]{graphicx}
\DeclareGraphicsExtensions{.bmp,.png,.pdf,.jpgg}
\usepackage[all]{xy}
\usepackage{verbatim}
\usepackage[mathcal]{euscript}
\usepackage{epsfig}
\usepackage{multicol}
\usepackage{color}
\usepackage{verbatim}
\usepackage{graphicx,float}

\date{\today}

\swapnumbers
\theoremstyle{plain}
\newtheorem{teo}{Theorem}[section]
\newtheorem{lema}[teo]{Lemma}
\newtheorem{prop}[teo]{Proposition}
\newtheorem{coro}[teo]{Corollary}

\theoremstyle{definition}
\newtheorem{defi}[teo]{Definition}
\newtheorem{obs}[teo]{Remark}
\newtheorem{obss}[teo]{Remarks}
\newtheorem{ejem}[teo]{Example}

\def\End{\mathop{\rm End_{\mathcal{C}_m}}\nolimits}
\def\C{\mathop{\mathcal{C}}\nolimits}

\input{xy}
\xyoption{poly}
\xyoption{2cell}
\xyoption{all}

\usepackage{url}

\title[The coloured mutation class of $\mathbb{A}_n$ quivers]{The coloured mutation class of $\mathbb{A}_n$ quivers}

\author[V. Gubitosi]{Viviana Gubitosi}
\address{Instituto de Matem\'{a}tica y Estad\'{\i}stica Rafael Laguardia, Facultad de Ingenier\'{\i}a - UdelaR, Montevideo, Uruguay, 11200 }
\email{gubitosi@fing.edu.uy}

\author[R. Parra]{Rafael Parra}
\address{Instituto de Matem\'{a}tica y Estad\'{\i}stica Rafael Laguardia, Facultad de Ingenier\'{\i}a - UdelaR, Montevideo, Uruguay, 11200 }
\email{rparra@fing.edu.uy}

\author[C. Qureshi]{Claudio Qureshi}
\address{Instituto de Matem\'{a}tica y Estad\'{\i}stica Rafael Laguardia, Facultad de Ingenier\'{\i}a - UdelaR, Montevideo, Uruguay, 11200 }
\email{cqureshi@fing.edu.uy}


\begin{document}
\maketitle

\begin{abstract}
In this paper we give an explicit and pure combinatorial description of the $m$-coloured quivers that appears in the $m$-coloured mutation class of a quiver of type $\mathbb{A}_n$. The $m$-coloured mutation defined  by Buan and Thomas in \cite{BT} generalizes the well-known quiver mutation of Fomin and Zelevinsky \cite{FZ}. In particular, our description  generalizes a result of Buan and Vatne, \cite{BV}, which we recover when $m=1$.
\end{abstract}

%

\section*{Introduction}

Given an integer $m\geq 1$, $m$-coloured quivers or coloured quivers were defined  in 2008 by Buan and Thomas  \cite{BT}. A coloured quiver $Q$ is a finite quiver  whose arrows have an associated colour $c$ in $\{0,\dots,m\}$ and such that $Q$ contains no loops and satisfies two additional properties called monochromaticity and skew-symmetry. Every acyclic quiver $Q$ can be seen as a coloured quiver, by regarding each arrow of $Q$ as an arrow of colour $0$, and then adding an arrow of color $m$ in the opposite direction.

On such coloured quivers, they also define an operation called coloured quiver mutation 
and show that it is compatible with mutation of $m$-cluster tilting objects. When $m=1$, the coloured mutation of the induced coloured quiver corresponds to the mutation defined by Fomin and Zelevinsky in \cite{FZ} and denoted by FZ-mutation. Since its definition quiver mutation has been the subject of several investigations; see, for instance \cite{BMR,BPRS,BR,BTo}. The problem of describing the FZ-mutation classes of quivers of different types was addressed by Vatne for type $\mathbb{D}_n$\cite{Va}, by Bastian for type $\mathbb{\tilde{A}}_n$ \cite{Ba} and by Buan and Vatne for type $\mathbb{A}_n$ \cite{BV}.  We recover this last result by taking $m=1$.\\


From an algebraic point of view, coloured quiver mutation gives some information on the $m$-cluster tilted algebras, i.e., algebras of the form $\End(T)$, for $T$ an $m$-cluster tilting object in an $m$-cluster category $\C_m$. Roughly speaking, given an hereditary finite dimensional algebra $H$ over an algebraically closed field, the $m$-cluster category $\C_m$ is obtained from the derived category $\mathcal{D}^b(H)$ by identifying the composition $[m]:=[1]^m$ of the shift functor $[1]$ with the Auslander - Reiten translation $\tau$.  Buan and Thomas assigned to every $m$-cluster tilting object $T$ a coloured quiver $Q_T$ where the vertices are in correspondence with the indecomposable direct summands of $T$ and  to determine the coloured arrows, they use the exchange triangles. It is known that any $m$-cluster tilting object can be reached from any other $m$-cluster tilting object by a sequence of exchanges \cite{Zhou}. Thomas, in \cite{T07}, showed that $m$-cluster tilting objects  are in bijective correspondence with the $m$-clusters associated with a finite root system  by Fomin and Reading in \cite{FR05}.


It is also observed in \cite{BT}  that the $0$-coloured part of the quiver $Q_T$ coincides with the Gabriel quiver of the $m$-cluster-tilted algebra $\End(T)$. Therefore, as a consequence we have that  Gabriel  quivers of  $m$-cluster-tilted algebras
can be combinatorially determined by applying repeated coloured mutations. See \cite[ Corollary 7.2]{BT}.\\

Two coloured quivers $Q_1$ and $Q_2$ are mutation equivalent if $Q_1$ can be obtained from $Q_2$ by some sequence of coloured mutations, and viceversa. An equivalence class is called a coloured mutation class. It was proved by Torkildsen  \cite{Tork}  that the coloured mutation class of a connected acyclic quiver $Q$ is finite if and only if $Q$ is either of Dynkin or extended Dynkin type, or has at most two vertices. In addition, for Dynkin type $\mathbb{A}_n$,  Torkildsen has found a formula for the number of elements in the mutation class, using a connection to the classical cell-growth problem \cite{Tork2}.\\


The aim of this paper is to give an explicit description of the coloured mutation class of $\mathbb{A}_n$ coloured quivers. That is, we will present the set of coloured quivers which can be obtained by iterating coloured mutations on a coloured quiver whose underlying graph is of Dynkin type $\mathbb{A}_n$. It turns out that the quivers in this mutation class are easily recognisable. \

In Definition \ref{la clase} we present a class $\mathcal{Q}_n^m$ of $m$-coloured quivers with $n$ vertices which includes all colourations of a quiver of type $\mathbb{A}_n$. The main result of this paper can be stated as follows:

\subsection*{Theorem A}\textit{ A connected  $m$-coloured quiver $Q$ is mutation equivalent to  $\mathbb{A}_n$ if and only if  $Q$ belongs to the class $\mathcal{Q}_n^m$. }

\medskip

In particular, specializing to the case $m=1$, we recover known results of \cite{BV}; and in addition, our description of the coloured mutation class of $\mathbb{A}_n$ coloured quivers gives a complete description of quivers of $m$-cluster-tilted algebras of type $\mathbb{A}_n$. These quivers are already known, see \cite{Murphy}, but the method used in this paper to obtain the description is purely combinatorial, and no prerequisites are needed.\\

This paper is organized as follows. After a preliminary section,
in which we fix the notations and recall some concepts needed
later, section 2 is devoted to recall  the definition of coloured quivers and coloured  mutation. In section 3, we describe the special class of coloured quivers with $n$ vertices, denoted by $\mathcal{Q}_n^m$,  which will turning  out the coloured mutation class of type $\mathbb{A}_n$.  We finish section 3  illustrating one of the most noticeable differences between the FZ-mutation (i.e., a $1$-coloured mutation) and the $m$-coloured mutation (for $m\geq 2$). Namely, the FZ-mutation is an involution however the $m$-coloured mutation is not even invertible in general. Section 4 is dedicated to the proof of our main result. In sections 5 and 6 we finish with some consequences, among which we recover the known results mentioned above.


\section{Preliminaries}

A  \textit{quiver} (or digraph) $Q$ is the data of two sets, $Q_0$ (the vertices) and $Q_1$ (the arrows); and two maps $s, t : Q_1 \rightarrow Q_0$ that assign to each arrow $\alpha$ its source $s(\alpha)$ and its target $t(\alpha)$. We write $\alpha: s(\alpha)\longrightarrow   t(\alpha)$ for the arrow $\alpha$ from $s(\alpha)$ to $t(\alpha)$. If either $s(\alpha)=i$ or $t(\alpha)=i$, we say that $\alpha$ is incident with the  vertex $i$. For any vertex $i$ in $Q_0$, the \textit{valency} of $i$ (in $Q$) is the number of neighbouring vertices,  i.e., the number of vertices $j\neq i$ such that there exists an arrow $\alpha \in Q_1$ with either $s(\alpha)=i$ and  $t(\alpha)=j$ or $s(\alpha)=j$ and $t(\alpha)=i$.\\

We say that $Q$ is \textit{simple} if there is at most one arrow between two distinct vertex.  In this case, each arrow $\alpha$ is determined by its source $s(\alpha)=i$ and its target $t(\alpha)=j$ and it is usual to denote $\alpha = ij$. A \textit{path} of length $k\geq 0$ in $Q$ is a sequence of arrows $\alpha_1\alpha_2\cdots \alpha_{k}$ such that $\alpha_i=x_i \rightarrow x_{i+1}\in Q_1$ for $1\leq i \leq k$.  A path is called {\it simple} when the vertices $x_1, x_2, \ldots, x_{k+1}$ are pairwise distinct. The path is called {\it closed} when $x_1=x_{k+1}$. A \textit{$k$-cycle} ($k\geq 3$) in  $Q$  is a closed path $\alpha_1\alpha_2\cdots \alpha_{k}$; i.e., $s(\alpha_1)=t(\alpha_{k})$ such that $\alpha_1\alpha_2\cdots \alpha_{k-1}$ is a simple path.
In the case where $Q$ is a simple quiver, we denote by $x_1x_2\cdots x_{k+1}$ the path $\alpha_1\alpha_2\cdots \alpha_{k}$ and by $(x_1x_2\cdots x_{k})$ the $k$-cycle $x_1x_2\cdots x_{k}x_1$. \\

Remember that a subquiver $Q'$ of $Q$ is called {\it induced} if every $\alpha \in Q_1$ such that $s(\alpha),t(\alpha)\in Q'_{0}$ satisfies $\alpha \in Q'_1$.  An \textit{induced cycle} is a cycle which is also an induced subquiver. A \textit{hole} is an induced cycle of length at least four. A quiver $Q$ is called \textit{free-hole}  if it does not contains holes. \

Let $I=\{x_1,\ldots,x_k\} \subseteq Q_0$. We denote by $Q[x_1,\ldots,x_k]$ (or just by $Q[I]$) the subquiver of $Q$ induced by $I$.  A \textit{complete} quiver is a quiver in which every pair of distinct vertices is connected by a pair of unique arrows (one in each direction).

\begin{figure}[H]

\[
\begin{array}{ccc}
  \xymatrix@R=30pt@C=15pt{ &&& . \ar@<0.3ex>[rd] \ar@<0.3ex>[dl] &&  &   \\   && . \ar@<0.3ex>[ur] \ar@<0.3ex>[rr] && . \ar@<0.3ex>[ll] \ar@<0.3ex>[lu] } &  & \xymatrix@R=30pt@C=15pt{ . \ar@<0.3ex>[d] \ar@<0.3ex>[rr] \ar@<0.3ex>[rrd] && . \ar@<0.3ex>[ll] \ar@<0.3ex>[dll] \ar@<0.3ex>[d]  \\   . \ar@<0.3ex>[urr] \ar@<0.3ex>[rr] \ar@<0.3ex>[u] &&  . \ar@<0.3ex>[u] \ar@<0.3ex>[ull] \ar@<0.3ex>[ll] }
\end{array}
\]

\caption{Complete quivers  with $3$ and $4$ vertices respectively.}
\end{figure}
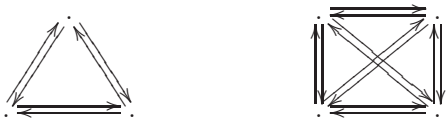

A \textit{clique} is a subset of vertices $I\subseteq Q_0$ such that $Q[I]$ is a complete quiver. A clique with $k$ vertices is called a $k$-clique. In the following we will identify, as usual, the $k$-clique $K$ with its induced complete quiver $Q[K]$. \\

A maximum clique in $Q$ is a clique with the maximum number of vertices and the \textit{clique number} of $Q$, denoted by $\omega(Q)$, is the number of vertices in a maximum clique.

%
%


\section{ Coloured quivers and coloured mutations}
We recall in this section the definitions of  coloured quivers and coloured quivers mutation given by Buan and Thomas in \cite{BT}.

\subsection{Coloured quivers}

Let $m$ be a positive integer. An $m$-coloured quiver  $Q$ is the data of two sets, $Q_0$ (the vertices) and $Q_1$ (the arrows); two maps $s, t : Q_1 \rightarrow Q_0$ that assign to each arrow $\alpha$ its source $s(\alpha)$ and its target $t(\alpha)$; and a colouration function $\kappa:Q_1 \rightarrow \{0,1,\ldots, m\}$ which associates to each arrow  $\alpha$ the colour $\kappa(\alpha)$. We write $\alpha: s(\alpha)\overset{(c)}{\longrightarrow}  t(\alpha)$ for the arrow $\alpha$ from $s(\alpha)$ to $t(\alpha)$ of colour $\kappa(\alpha)=c$.\\



Let $q_{ij}^{(c)}$ denote the number of arrows from
$i$ to $j$ of colour $c$. We will say that $Q$ has no loops if $q_{ii}^{(c)} = 0$ for all $c$. The quiver  $Q$ is said to be monochromatic if $q_{ij}^{(c)} \neq 0$, then $q_{ij}^{(c')} = 0$ for $c \neq c'$. Clearly every simple quiver is monochromatic. Finally, we say that $Q$ is skew-symmetric if $q_{ij}^{(c)} = q_{ji}^{(m-c)}$. \\

From now on, we will consider coloured quivers with the above three additional conditions.\

%
%
%
%
%

\subsection{Coloured  mutation}

In such a  coloured quiver $Q$,  we have the following operation $\mu_j$ called coloured quiver mutation at  $j$.  Let $j$ be a vertex in $Q$ and let $\mu_j(Q) = \widetilde{Q}$ be the coloured quiver such that

$$\tilde{q}_{ik}^{(c)} =
\begin{cases}  q_{ik}^{(c+1)} & \text{  if $j =k$} \\
		 q_{ik}^{(c-1)} &\text{  if $j=i$} \\
		 \max \{0, q_{ik}^{(c)} - \sum_{t \neq c} q_{ik}^{(t)} + (q_{ij}^{(c)} - q_{ij}^{(c-1)}) q_{jk}^{(0)}
		 + q_{ij}^{(m)} (q_{jk}^{(c)}  -q_{jk}^{(c+1)}) \} & \text{  if $i \neq j \neq k$}
                 \end{cases}
$$

They also  give an alternative description of coloured quiver mutation at vertex $j$ that we recall here.

\subsubsection*{  \hspace*{4cm}    Alternative algorithm for coloured mutation:}

\begin{enumerate}
\item For each pair of arrows
$$ \xymatrix {i \ar^{(c)}[r] & j\ar^{(0)}[r] &k }$$
with $i\ne k$, the arrow from $i$ to $j$ of arbitrary colour $c$, and the arrow from
$j$ to $k$ of colour $0$, add a pair of arrows: an arrow from $i$ to $k$
of colour $c$, and one from $k$ to $i$ of colour $m-c$.
\item If the quiver is not longer monochromatic, because for some pair of vertices $i$ and $k$ there are arrows from
$i$ to $k$ which have two different colours, cancel the same number of
arrows of each colour, until the  monochromaticity property is satisfied.
\item Add one to the colour of any arrow going into $j$ and subtract one
from the colour of any arrow going out of $j$.
\end{enumerate}

Note that  the operations over the colours performed at step $(3)$ have to be done modulo $m+1$.\\

Buan and Thomas proved that the above algorithm is well-defined and
correctly calculates coloured
quiver mutation as previously defined by themselves.\\


\subsection{Mutation class of type $\mathbb{A}_n$}

We will say that the underlying graph of a coloured quiver  $Q$ (with no loops, monochromatic and  skew-symmetric)  is the graph obtained by keeping one edge  $\xymatrix{i \ar@{-}[r] & j}$ for each pair of arrows
$\xymatrix{i \ar@<0.6ex>^{(c)}[r] & j \ar@<0.6ex>^{(m-c)}[l]}$ of the coloured quiver $Q$.\\

In the following, we will consider coloured quivers whose underlying graph is a Dynkin graph of type $\mathbb{A}_n$.
We will refer to these quivers as coloured quivers of type $\mathbb{A}_n$ or $\mathbb{A}_n$-quivers for short.\\

\begin{figure}[H]
\begin{center}
\[\xymatrix{1 \ar@{-}[r] & 2 \ar@{-}[r] & \cdots \ar@{-}[r] & n-1 \ar@{-}[r] & n}
\]
\caption{The Dynkin graph   $\mathbb{A}_n$}
\label{quiverAn}
\end{center}
\end{figure}

Two quivers are said to be mutation equivalent if one can be obtained from the other by some sequence of coloured mutations, and viceversa. An equivalence class will be called a mutation class. \


Let $Q$ be the following  coloured quiver with arrows only of colour $0$ and $m$ whose underlying graph is a Dynkin graph of type $\mathbb{A}_n$.

$$\xymatrix{1 \ar@<0.6ex>^{(0)}[r] & 2 \ar@<0.6ex>^{(m)}[l] \ar@<0.6ex>^{(0)}[r] & 3  \ar@<0.6ex>^{(m)}[l]  \ar@<0.6ex>^{(0)}[r] & \cdots   \ar@<0.6ex>^{(m)}[l] \ar@<0.6ex>^{(0)}[r] & n  \ar@<0.6ex>^{(m)}[l]  }$$

We call the mutation class of type $\mathbb{A}_n$ to the set of all quivers
mutation equivalent to $Q$.\
By a result of Torkildsen \cite{Tork, Tork2}, we know that this mutation class is finite. \

%

\begin{ejem}  The following figure  shows all non-isomorphic coloured quivers in the mutation class of type $\mathbb{A}_3$ for $m=2$.

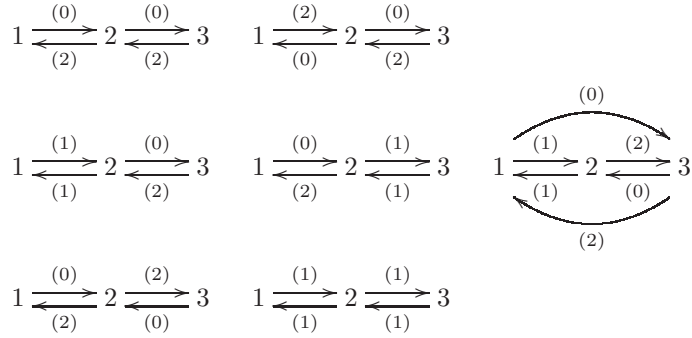
\begin{figure}[H]
\begin{center}
$$\begin{array}{ccc}
  $\xymatrix{1 \ar@<0.6ex>^{(0)}[r] & 2 \ar@<0.6ex>^{(2)}[l] \ar@<0.6ex>^{(0)}[r] & 3  \ar@<0.6ex>^{(2)}[l] }$ & $\xymatrix{1 \ar@<0.6ex>^{(2)}[r] & 2 \ar@<0.6ex>^{(0)}[l] \ar@<0.6ex>^{(0)}[r] & 3  \ar@<0.6ex>^{(2)}[l] }$ &  \\
  $\xymatrix{1 \ar@<0.6ex>^{(1)}[r] & 2 \ar@<0.6ex>^{(1)}[l] \ar@<0.6ex>^{(0)}[r] & 3  \ar@<0.6ex>^{(2)}[l] }$ & $\xymatrix{1 \ar@<0.6ex>^{(0)}[r] & 2 \ar@<0.6ex>^{(2)}[l] \ar@<0.6ex>^{(1)}[r] & 3  \ar@<0.6ex>^{(1)}[l] }$ & $\xymatrix{1 \ar@<0.6ex>^{(1)}[r] \ar@/^5mm/@<1.6ex>^{(0)}[rr]& 2 \ar@<0.6ex>^{(1)}[l] \ar@<0.6ex>^{(2)}[r] & 3  \ar@<0.6ex>^{(0)}[l]  \ar@/^5mm/@<1.6ex>^{(2)}[ll]}$ \\
  $\xymatrix{1 \ar@<0.6ex>^{(0)}[r] & 2 \ar@<0.6ex>^{(2)}[l] \ar@<0.6ex>^{(2)}[r] & 3  \ar@<0.6ex>^{(0)}[l] }$ & $\xymatrix{1 \ar@<0.6ex>^{(1)}[r] & 2 \ar@<0.6ex>^{(1)}[l] \ar@<0.6ex>^{(1)}[r] & 3  \ar@<0.6ex>^{(1)}[l] }$ &
\end{array}$$
\caption{The mutation class of the 2-coloured quiver $\mathbb{A}_3 $}
\label{mutation class A3}
\end{center}
\end{figure}

\end{ejem}

\section{The set $\mathcal{Q}_n^m$}


In this section we define a special class of $m$-coloured quivers with $n$ vertices which will turning  out the coloured mutation class of type $\mathbb{A}_n$.\\

Remember that  a clique is a directed quiver in which every pair of distinct vertices is connected by an unique pair of edges (one in each direction). We denote by $\mathcal{C}_r$ a clique with $r$ vertices. \\

\begin{defi}\label{la clase}

Let $\mathcal{Q}_n^m$ be the class of $m$-coloured simple and connected quivers $Q$ with $n$ vertices and no holes which satisfy the following two additional  conditions:

\begin{enumerate}
\item For each vertex $v$ in $(Q_n^m)_0$ with $z\geq 1$ neighbours, there exists two cliques $\mathcal{C}_r$  and $\mathcal{C}_k$ such that $v \in (\mathcal{C}_r)_0 \cap (\mathcal{C}_k)_0$,  $r+k=z+2$ and $r,k \leq m+2$. In addition, there are not arrows between two vertices $i\in (\mathcal{C}_r)_0 $ and  $j\in (\mathcal{C}_k)_0$.

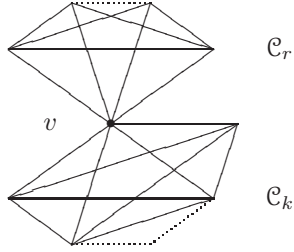
\begin{figure}[H]
\begin{center}

\[
\xy/r4pc/:{\xypolygon10"A"{~<{}~>{}{}}}
*+{{\scriptstyle \bullet}},

\POS"A6" \drop{\begin{array}{llllll} &&&&& v \end{array}}
\POS"A10" \ar@{-}   "A1"
\POS"A9" \ar@{.}   "A10"
\POS"A8" \ar@{.}   "A9"
\POS"A2" \ar@{-}   "A3"
\POS"A7" \ar@{-}   "A8"
\POS"A7" \ar@{-}   "A1"
\POS"A4" \ar@{-}   "A5"
\POS"A3" \ar@{.}   "A4"
\POS"A7" \ar@{-}   "A0"
\POS"A4" \ar@{-}   "A0"
\POS"A2" \ar@{-}   "A0"
\POS"A1" \ar@{-}   "A0"
\POS"A8" \ar@{-}   "A0"
\POS"A10" \ar@{-}   "A0"
\POS"A5" \ar@{-}   "A2"
\POS"A0" \ar@{-}   "A5"
\POS"A0" \ar@{-}   "A3"
\POS"A8" \ar@{-}   "A10"
\POS"A8" \ar@{-}   "A1"
\POS"A7" \ar@{-}   "A10"
\POS"A2" \ar@{-}   "A4"
\POS"A3" \ar@{-}   "A5"
\POS"A2" \drop{\begin{array}{llllll} &&&&& \mathcal{C}_r  \end{array}}
\POS"A10" \drop{\begin{array}{llllll} &&&&& \mathcal{C}_k  \end{array}}

\endxy   \]

\end{center}
\end{figure}

%

\item For each triangle $$
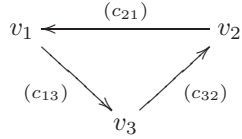
\xymatrix{ v_1 \ar_{(c_{13})}[rd] & & v_2 \ar_{(c_{21})}[ll] \\  &  v_3  \ar_{(c_{32})}[ru] &  }$$

with $v_1,v_2,v_3$ three different vertices belonging to the same clique $\mathcal{C}$ we have that  $$c_{21}+c_{13}+c_{32}=m-1 \text{\ \ or \ \ } (m-c_{21})+(m-c_{13})+(m-c_{32})=m-1.$$

\end{enumerate}

\end{defi}
\vspace*{.5cm}

\begin{obss}\label{Remark color triangulo}

Observe that we do not compute the sum of the colours modulo $m+1$. If a given orientation of the triangle gives $c_{21}+c_{13}+c_{32}=m-1$, the opposite orientation gives $c_{12}+c_{31}+c_{23}=2m+1$.\\

If $c_{13}$ is the smaller of the colours $\{c_{21},c_{13},c_{32},c_{12},c_{31},c_{23}\}$, then the orientation of the triangle induced by  $c_{13}$ is the orientation whose sum of the colours is  $m-1$.
\end{obss}

\begin{obs}
If $Q \in \mathcal{Q}_n^m$, $Q$ is a simple quiver. Then,  for each pair of vertices  $i$ and  $j$  in $Q_0$  there is at most one pair of arrows $\xymatrix{i \ar@<0.6ex>^{(c_{ij})}[r] & j\ar@<0.6ex>^{(m-c_{ij})}[l] }$  with  $c_{ij}\in \{0, 1, \dots, m \}$. 
\end{obs}

Since all our coloured quivers are skew-symmetric, the colour of every arrow $\xymatrix{v_i \ar[r]^{(c_{ij})} & v_j}$ determines the colour of the arrow $\xymatrix{v_j \ar[r]^{(c_{ji})} & v_i}$ according to the  equation $c_{ji}=m-c_{ij}$. Then, it will cause no confusion if  we  simply draw $\xymatrix{v_i \ar[r]^{c_{ij}} & v_j}$ instead of $\xymatrix{v_i \ar@<0.6ex>^{(c_{ij})}[r] & v_j\ar@<0.6ex>^{(m-c_{ij})}[l] }$.\\

\begin{figure}[H]
\begin{center}
\[
\xymatrix@R=25pt@C=15pt{ \cdot \ar[rr]^{1} && \cdot \ar[d]^0 && &&  & \cdot \ar[dr]^0 & \\
\cdot \ar[u]^2 \ar[rru]^{0} && \cdot \ar[ll]^1 \ar[rr]^0 \ar[llu]^{ 0 } \ar[rrd]_1  && \cdot \ar[rr]^2 \ar[d]^0 && \cdot \ar[ru]^1 && \cdot \ar[ll]^0 \\
&& \cdot \ar[u]^0 \ar[rru]^1  && \cdot  \ar[ll]^0 \ar[rr]^0 && \cdot \ar[rr]^1 \ar[dl]^0 && \cdot \\
&& && & \cdot \ar[lu]^1 & && 
}
\]

\caption{A $2$-coloured quiver in the class $\mathcal{Q}_{13}^2$ }
\label{Q132}
\end{center}
\end{figure}
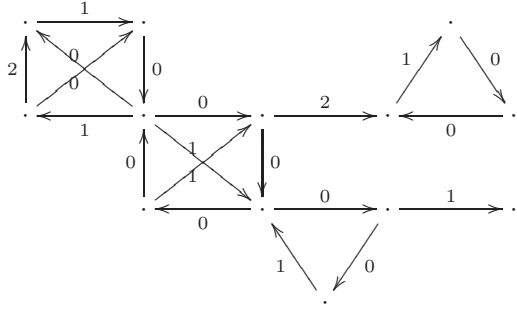

Here are some elementary properties of this class.\\

\begin{lema} \label{triangulos admisibles}
Let $v$, $v_1$ and $v_2$ be part of a triangle belonging to the class $\mathcal{Q}_n^m$. We label the colours of the arrows as in the following diagram:

$$
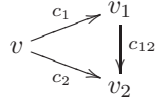
\xymatrix@R=4pt@C=10pt{ && v_1 \ar[dd]^{c_{12}}\\
v \ar[rru]^{c_1} \ar[rrd]_{c_2}\\
&& v_2}$$

Then:

\begin{enumerate}
  \item $c_1\neq c_2$.
  \item If $c_1< c_2$ then $c_{12}=c_2-c_1-1$.

\end{enumerate}

\end{lema}

\begin{proof}
Without loss of generality we can assume that $c_1+c_{12}+m-c_2=m-1$. In particular, we have that $0 \leq c_1,c_{12},m-c_2< m$. If $c_1=c_2$ then  $c_{12}=-1$, a contradiction. Moreover, if we assume that $c_1< c_2$, then $c_{12}=m-1-c_1-(m-c_2)=c_2-c_1-1$.
\end{proof}

\begin{lema} \label{completo mas grande}
Let $v$, $w_1, \cdots, w_{m+1}$ the $m+2$ vertices of a $(m+2)$-clique $\mathcal{C}_{m+2} $ in $\mathcal{Q}_n^m$. We label the colours of the arrows leaving $v$ as in the following diagram:

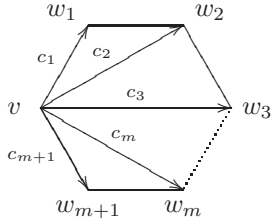
\[
\xy/r3pc/:{\xypolygon6"A"{~<{}~>{}{}}}

\POS"A6" \drop{ \begin{array}{c} \\  w_{m}\end{array}}
\POS"A2" \ar@{-} "A1"
\POS"A3" \ar@{-}   "A2"
\POS"A4" \ar@{->}  ^{c_1}   "A3"
\POS"A4" \ar@{->}  ^{c_2}   "A2"
\POS"A4" \ar@{->}  ^{c_3}   "A1"
\POS"A4" \ar@{->}  ^{c_m}   "A6"
\POS"A4" \ar@{->}  _{c_{m+1}}   "A5"
\POS"A4" \ar@{-}   "A5"
\POS"A5" \ar@{-}   "A6"
\POS"A6" \ar@{.}   "A1"
\POS"A5" \drop{ \begin{array}{c} \\  w_{m+1}\end{array}}
\POS"A4" \drop{\begin{array}{ccc} v &&  \end{array}}
\POS"A1" \drop{ \begin{array}{ccc} && w_{3}\end{array}}
\POS"A2" \drop{ \begin{array}{ccc} && w_{2}\\ && \end{array}}
\POS"A3" \drop{ \begin{array}{ccc} w_{1}&& \\ && \end{array}}
\endxy   \]

Then, given $c\in \{0,\cdots,m\} $ there is an unique colour $c_i$ such that $c_i=c$. In particular, there is an unique arrow $\xymatrix{v \ar[r]^{c_{i}} & v_i}$ of colour $0$.

\end{lema}

\begin{proof}
According to the previous lemma, all the $m+
1$ colours $c_i$ are distinct. The claim follows since there are exactly $m+1$ different colours in the set $\{0,\cdots,m\} $.
\end{proof}

The above lemma implies directly the following corollary.

\begin{coro}\label{ciclo color cero}
If $w_0, \cdots, w_{m+1}$  are the $m+2$ vertices of a clique $\mathcal{C}_{m+2} $ in $\mathcal{Q}_n^m$. Then, there exits a permutation $\theta$ of the numbers $0, \cdots, m+1$ such that the $(m+2)$-cycle  $w_{\theta(0)} \rightarrow w_{\theta(1)} \rightarrow \cdots \rightarrow w_{\theta(m+1)}\rightarrow w_{\theta(0)}$ has all its arrows coloured with zero.

\end{coro}


\vspace*{.5cm}

We continue with an important property of the class $\mathcal{Q}_n^m$. \\

\begin{lema}\label{cerrado por mutaciones}
The class $\mathcal{Q}_n^m$ is closed under  coloured quiver  mutation.
\end{lema}

\begin{proof}
Let $Q$ be a quiver in the class $\mathcal{Q}_n^m$ and  let $v$ be a vertex with $z$ neighbours, $z\leq 2m+2$. Then, there exist two cliques $\mathcal{C}_r$  and $\mathcal{C}_k$ such that $v\in \mathcal{C}_r$,  $v\in \mathcal{C}_k$, $r+k=z+2$ and $r,k \leq m+2$. Label $v_1, \cdots, v_{r-1}$ the neighbours of $v$ belonging to the clique $\mathcal{C}_r$  and label $w_1, \cdots, w_{k-1}$ the neighbours of $v$ belonging to the clique $\mathcal{C}_k$.
For each vertex $v_i$ the arrow $\xymatrix{v \ar[r]^{c_{i}} & v_i}$ has colour $0\leq c_i \leq m$ and for each vertex  $w_i$ the arrow $\xymatrix{v \ar[r]^{d_{i}} & w_i}$ has colour $0\leq d_i \leq m$. This colouration of the arrows leaving $v$ determines the colouration of the remaining  arrows. In fact, the condition $c_{ij}=c_j-c_i-1$ determines the colours of all the arrows $\xymatrix{v_i \ar[r]^{c_{ij}} & v_j}$ with $v_i,v_j\in (\mathcal{C}_r)_0$. Analogously the condition $d_{ij}=d_j-d_i-1$ determines the colours of all the arrows $\xymatrix{w_i \ar[r]^{d_{ij}} & w_j}$ with $w_i,w_j\in (\mathcal{C}_k)_0$.\\

By Lemma \ref{triangulos admisibles}, we know that all the colours $c_i$ are distinct. The same holds for the colours $d_i$. Hence, we can assume that $c_1<c_2<\cdots < c_{r-1} $ and $d_1<d_2<\cdots < d_{k-1} $ . Therefore, a colour $c_i$ can possibly coincide with at most another colour $d_j$. \\

Let $Q'$ be the quiver obtained after mutating $Q$ at the vertex $v$. We will show that $Q'$ belongs to the class  $\mathcal{Q}_n^m$. \\

On account of the previous considerations about the colours the proof falls naturally into three cases.

  \begin{enumerate}

  \item \underline{None of the  colours $c_i$ or $d_i$ is zero}:
  In this case, the mutation at the vertex $v$ simply changes the colour of the arrows arriving or leaving $v$  and preserves without changes all the others colours, as follows:

  \[\begin{array}{ cccc}
      \quad  
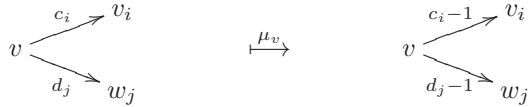
\xymatrix@R=4pt@C=10pt{ &&  v_i \\  v \ar[rru]^{c_i} \ar[rrd]_{d_j } && \\ &&  w_j }  \quad

    & \xymatrix@R=4pt@C=10pt{ & & \\ & \stackrel{\mu_v}{\longmapsto} & \\ &  &}&  \quad  \xymatrix@R=4pt@C=10pt{ &&  v_i \\  v \ar[rru]^{c_{i}-1} \ar[rrd]_{d_{j}-1} && \\ &&  w_j } \quad
\end{array}\]

\quad

  Every triangle $\xymatrix@R=4pt@C=10pt{ && v_j \ar[dd]^{c_{ji}}\\ v \ar[rru]^{c_j} \ar[rrd]_{c_i}\\ && v_i}$  of $Q$ changes to the triangle $\xymatrix@R=4pt@C=10pt{ && v_j \ar[dd]^{c_{ji}}\\ v \ar[rru]^{c_j-1} \ar[rrd]_{c_i-1}\\ && v_i}$.

  \quad

  Assume that $c_{j}+c_{ji}+m-c_{i}=m-1 $, then the new triangle  also satisfies the condition  $c_{j}-1+c_{ji}+m-(c_{i}-1)=m-1 $. The same reasoning applies to the triangles in the clique $\mathcal{C}_k$.

  \item  \underline{There is an arrow of colour zero in only one clique}:
  Without loss of generality we can assume that $c_1=0$. Since $c_1$ belongs to the clique $\mathcal{C}_r$ and we are in the case where only one clique contains an arrow leaving $v$ of colour $0$, Lemma \ref{completo mas grande} implies that the  clique $\mathcal{C}_k$ satisfies that $k < m+1$.  Then, the mutation at the vertex $v$ acts as follows:

  \[\begin{array}{ cccc}
      \quad  
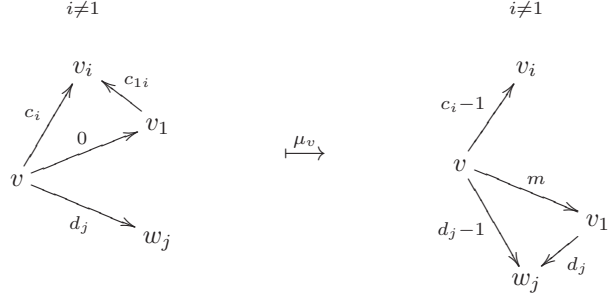
\xymatrix@R=10pt@C=10pt{  & \scriptstyle{i\neq 1} & \\& v_i & \\ &&  v_1 \ar[lu]_{c_{1i}}\\  v  \ar[ruu]^{c_i}\ar[rru]^{0} \ar[rrd]_{d_j } && \\ &&  w_j }  \quad

    & \xymatrix@R=4pt@C=10pt{  && \\ && \\& & \\ && \\ && \\ & \stackrel{\mu_v}{\longmapsto} & \\ &  &}&  \quad
    \xymatrix@R=10pt@C=10pt{  & \scriptstyle{i\neq 1} & \\& v_i & \\ &&   \\  v  \ar[ruu]^{c_i-1} \ar[rrd]^{m} \ar[rdd]_{d_{j}-1} &&  \\ &&v_1\ar[dl]^{d_j}\\ &  w_j& }   \quad
\end{array}\]

For each arrow $\xymatrix{v \ar[r]^{c_{i} }& v_i}$ in $Q$ ($i\neq 1$), the mutation adds a new  arrow  $\xymatrix{v_1 \ar[r]^{c_{i}} & v_i}$  which cancels the arrow $\xymatrix{v_1\ar[r]^{c_{1i}} & v_i}$ since $c_{1i}=c_i-1\neq c_i$.

Then,  each triangle $$
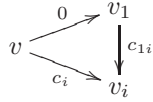
\xymatrix@R=4pt@C=10pt{ && v_1 \ar[dd]^{c_{1i}}\\ v \ar[rru]^{0} \ar[rrd]_{c_i}\\ && v_i}$$  of $Q$  disappears in $Q'$, and consequently the clique $\mathcal{C}_r$ of $Q$ becomes a clique $\mathcal{C'}_{r-1}$ with $r-1 $ vertices on $Q'$.

In addition for each arrow $\xymatrix{v \ar[r]^{d_{j} }& w_j}$ in $Q$, the mutation creates a new  triangle  $$
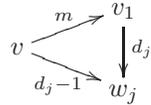
\xymatrix@R=4pt@C=10pt{ && v_1 \ar[dd]^{d_j}\\ v \ar[rru]^{m} \ar[rrd]_{d_j-1}\\ && w_j}$$ in $Q'$. Therefore the clique $\mathcal{C}_k$ of $Q$ becomes a clique $\mathcal{C'}_{k+1}$ with $k+1 $ vertices in $Q'$ and $k+1 < m+2$. We have proved more, namely that despite the fact that the number of  neighbours of the vertices $w_j$ is increased by one after mutating at $v$, it is still bounded by $2m+2$.  Let $z$ be the number of neighbours that  the vertex $v_1$ has in $Q$. Then, the number of  neighbours of the vertex $v_1$ in $Q'$ changes to $z-r+1+k$. Since $z-r < m$  (by condition 1 of  Definition \ref{la clase}) we have that $z-r+1+k\leq 2m+2$.  Finally observe that, after mutating at $v$,  the number of  neighbours of the vertex $v$  remains unchanged and the number of  neighbours of the vertices $v_i$ ($i\neq 1$) is reduced by one.

\item  \underline{There exits an arrow of colour zero in both cliques}: As in the above case, we can assume that $c_1=d_1=0$.  Then, the mutation at the vertex $v$ acts as follows:

  \[\begin{array}{ cccc}
      \quad  
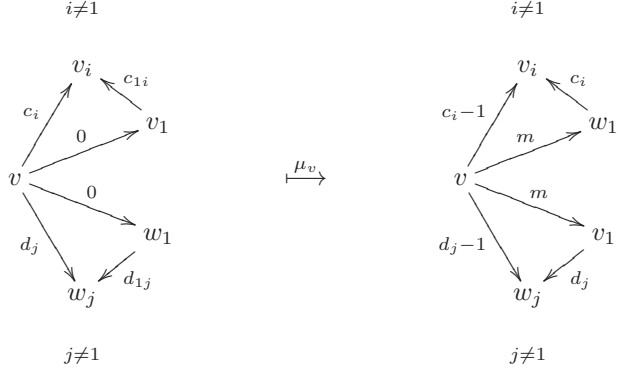
\xymatrix@R=10pt@C=10pt{  & \scriptstyle{i\neq 1} & \\& v_i & \\ &&  v_1 \ar[lu]_{c_{1i}}\\  v  \ar[rru]^{0} \ar[ruu]^{c_{i}} \ar[rrd]^{0} \ar[rdd]_{d_{j}} &&  \\ &&w_1\ar[dl]^{d_{1j}}\\ &  w_j& \\  & \scriptstyle{j\neq 1} &}  \quad

    & \xymatrix@R=4pt@C=10pt{ & & \\ && \\ && \\ & & \\ && \\ && \\ & \stackrel{\mu_v}{\longmapsto} & \\ &  &}&  \quad
    \xymatrix@R=10pt@C=10pt{  & \scriptstyle{i\neq 1} & \\& v_i & \\ &&  w_1\ar[lu]_{c_i} \\  v  \ar[rru]^{m} \ar[ruu]^{c_i-1} \ar[rrd]^{m} \ar[rdd]_{d_j-1} &&  \\ &&v_1\ar[ld]^{d_j}\\ &  w_j& \\  & \scriptstyle{j\neq 1} & }   \quad
\end{array}\]

As in the previous case, for $i\neq 1$, the mutation at $v$ deletes the arrow  $\xymatrix{  v_1 \ar[r]^{c_{1i}} &  v}$ of the triangles
$$
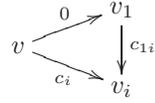
\xymatrix@R=4pt@C=10pt{ && v_1 \ar[dd]^{c_{1i}}\\ v \ar[rru]^{0} \ar[rrd]_{c_i}\\ && v_i}$$  of $Q$ and  for  $j\neq 1$, adds a new arrow  $\xymatrix{  v_1 \ar[r]^{d_j} &  w_j}$ creating a new triangle $$
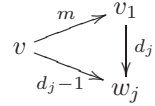
\xymatrix@R=4pt@C=10pt{ && v_1 \ar[dd]^{d_j}\\ v \ar[rru]^{m} \ar[rrd]_{d_{j}-1}\\ && w_j}$$ in $Q'$.  Likewise,  for $j\neq 1$, the mutation at $v$ deletes the arrow  $\xymatrix{  w_1 \ar[r]^{d_{1j}} &  w_j}$ of the triangles  $$
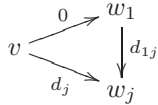
\xymatrix@R=4pt@C=10pt{ && w_1 \ar[dd]^{d_{1j}}\\ v \ar[rru]^{0} \ar[rrd]_{d_j}\\ && w_j}$$  of $Q$ and for $i\neq 1$, adds a new arrow  $\xymatrix{  w_1 \ar[r]^{c_i} &  v_i}$ creating a new triangle $$
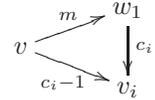
\xymatrix@R=4pt@C=10pt{ && w_1 \ar[dd]^{c_i}\\ v \ar[rru]^{m} \ar[rrd]_{c_i-1}\\ && v_i}$$ in $Q'$. Therefore, the vertices $v_1$  and $ w_1$ swap their places at the cliques  $\mathcal{C}_r$  and $\mathcal{C}_k$ respectively, but clearly the size of each clique do not change.

The proof is completed by showing that the number of neighbours of  the vertices $v_1$ and $w_1$ in $Q'$ remains bounded by $2m+2$.
Let $z$ be the number of neighbours that  the vertex $v_1$ has in $Q$. Then, the number of  neighbours of the vertex $v_1$ in $Q'$ changes to $z'=z-r+1+k-1$. Since $z-r \leq m$   we have that $z'\leq m+(m+2)=2m+2$. The same reasoning applies to the number of  neighbours of the vertex $w_1$ in $Q'$.

\end{enumerate}

From the analysis of the three cases above it also follows that the
quiver $Q'$ is simple and connected. To complete the proof, all that
remains is to prove that $Q'$ is hole-free. By contradiction, assume
that there is an induced $k$-cycle $c=(x_1x_2\cdots x_k)$ in $Q'$ with
$k\geq 4$. Since $Q$ is hole-free, $c$ must contain an arrow $e\in
Q'_1\setminus Q_1$. The only possibilities are $e=v_1w_j$ for some $j$
(if $c_1=0$) and $e=w_1v_i$ for some $i$ (if $d_1=0$). We can assume
without loss of generality that $c_1=0$ and $x_1x_2=v_1w_j$ with $1\leq j < k$ (in addition, $j\neq 1$ if $d_1=0$). Since $x_k\neq w_j$ and $x_k w_j \not\in Q'_1$ then $x_k
\not\in \{w_1,w_2,\ldots, w_{k-1}\}$. For $3\leq i <k$ we have $v_1x_i
\not\in Q'_1$ and then $x_i \not\in \{w_1,w_2,\ldots, w_{k-1}\}$.
Therefore, $v_1w_j$ is the only arrow in $c$ which is not in $Q$. Since
$v_1v$ and $w_jv$ are arrows of $Q'$, we have $v\not\in
\{x_3,\ldots,x_k\}$ and $c'=(x_1vx_2x_3\cdots x_k)$ would be an
induced cycle of length $k+1\geq 4$ in $Q$ which is a contradiction
because $Q$ is free-hole.

\end{proof}

If $m=1$ it is known that the mutation is invertible, because it is an involution (i.e., $\mu^2=Id$).
It is easy to see that for the coloured mutation with $m\geq 2$ is not longer true that $\mu^{m+1}=Id$. See Example \ref{mu3 no id} below. However, we have the following.

\begin{lema}\label{mutacion invertible}
The coloured quiver  mutation restricted to the  class $\mathcal{Q}_n^m$ is invertible. Moreover
if $Q\in \mathcal{Q}_n^m$ and $v$ is any vertex of $Q$ we have $\mu_v^{m+1}(Q)=Q$.
\end{lema}

\begin{proof}
Let $Q$ be a quiver belonging to the class $\mathcal{Q}_n^m$ and  let $v$ be any vertex of $Q$. We will show that $\mu_v^{m+1}(Q)=Q$.

Let $u_1, \cdots, u_{z}$ be  the neighbours of $v$. We can assume that indexes are arranged in such a way that  the colours $c_i$ of the arrows  $\xymatrix{v \ar[r]^{c_{i}} & u_i}$ satisfies $0 \leq c_1<c_2<\cdots < c_z \leq m$. At most there are two arrows $\xymatrix{v \ar[r]^{c_{i}} & u_i^1}$ and $\xymatrix{v \ar[r]^{c_{i}} & u_i^2}$ with the same colour $c_i$, then we can assume without loss of generality that all the colours $c_i$ are different. We know that, if $u_i$ and $u_j$ belong to the same complete subquiver, then we have an arrow $\xymatrix{u_i \ar[r]^{c_{ij}} & u_j}$ of colour $c_{ij}=c_{j}-c_i-1$ in $Q$. If $j<i$ we write $c_{j}-c_i=m-c_i+c_j+1 \in \{0,\cdots,m\}$.


We write $m+1=(c_1+1)+(c_2-c_1)+\cdots + (c_z-c_{z-1})+(m-c_z)$ which decomposes the mutation  $\mu_v^{m+1}$ into $z+1$ steps in such a way that each of them changes the shape of the resulting quiver. We will show that at the end of the $z+1$ steps we obtain the original quiver $Q$.

In the first step we apply $\mu_v^{c_1+1}$ which add an arrow $\xymatrix{u_1 \ar[r]^{c_{i}-c_1} & u_i}$ for all $i\neq 1$. Then, if the quiver $Q$ has already an arrow $\xymatrix{u_1 \ar[r]^{c_{i}-c_1-1} & u_i}$ it is deleted in the resulted quiver $\mu_v^{c_1+1}(Q)$. In addition, the arrows $\xymatrix{v \ar[r]^{c_{i}} & u_i}$ of $Q$ change to the colour $\xymatrix{v \ar[r]^{c_{i}-c_1-1} & u_i}$ in $\mu_v^{c_1+1}(Q)$.

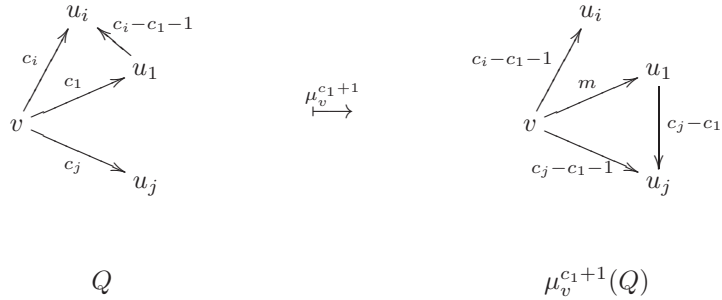
 \[\begin{array}{ cccc}
      \quad  \xymatrix@R=10pt@C=10pt{   & \\& u_i & \\ &&  u_1 \ar[lu]_{c_{i}-c_1-1}\\  v  \ar[ruu]^{c_i}\ar[rru]^{c_1} \ar[rrd]_{c_j } && \\ &&  u_j }  \quad

    & \xymatrix@R=4pt@C=10pt{  && \\ && \\& & \\ && \\ && \\ & \stackrel{\mu_v^{c_1+1}}{\longmapsto} & \\ &  &}&  \quad
    \xymatrix@R=10pt@C=10pt{   & \\& u_i & \\ &&  u_1 \ar[dd]^{c_j-c_1} \\  v  \ar[ruu]^{c_{i}-c_1-1}\ar[rru]^{m} \ar[rrd]_{c_j-c_1-1 } && \\ &&  u_j }  \quad \\
    && \\
    &&\\

    Q && \mu_v^{c_1+1}(Q)

\end{array}\]

%

When we arrive to step $i$ we apply  $\mu_v^{c_i-c_{i-1}}$ to the quiver  $\mu_v^{c_1+1+(c_2-c_1)+\cdots +(c_{i-1}-c_{i-2})}(Q)$. Therefore, if we had created an arrow $\xymatrix{u_1 \ar[r]^{c_{i}-c_1} & u_i}$ at step 1, this arrow is going to be canceled with the arrow  $\xymatrix{u_i \ar[rr]^{m-c_i+ c_{1}+1} && u_1}$ added at this step $i$. Otherwise, if we had deleted the arrow $\xymatrix{u_1 \ar[rr]^{c_{i}-c_1-1} && u_i}$ of $Q$ at step 1, now at step $i$ we are going to recover it (with the  original colour) in the quiver $\mu_v^{c_1+1+(c_2-c_1)+\cdots +(c_{i}-c_{i-1})}(Q)$.

The colour $c_j$ of the arrows $\xymatrix{v \ar[r]^{c_{j}} & u_j}$ of $Q$   change to the colour $c_{j}-c_i-1$ in the quiver $\mu_v^{c_1+1+(c_2-c_1)+\cdots +(c_{i}-c_{i-1})}(Q)$.


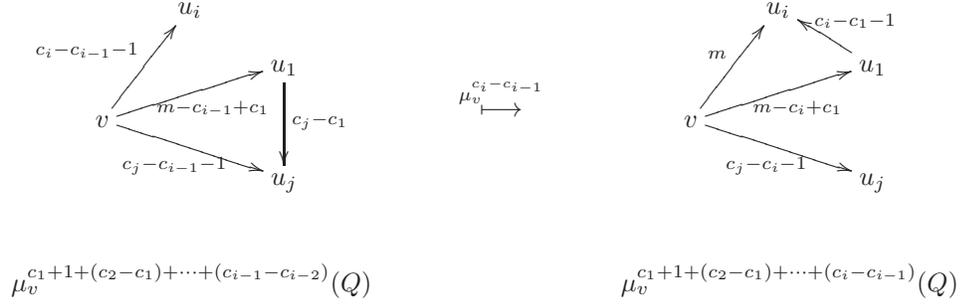
\[\begin{array}{ cccc}
      \quad  \xymatrix@R=10pt@C=20pt{   & \\& u_i & \\ &&  u_1 \ar[dd]^{c_j-c_1} \\  v  \ar[ruu]^{c_{i}-c_{i-1}-1}  \ar[rru]_{m-c_{i-1}+c_1} \ar[rrd]_{c_j-c_{i-1}-1 } && \\ &&  u_j }  \quad

    & \xymatrix@R=4pt@C=10pt{  && \\ && \\& & \\ && \\ && \\ & \stackrel{\mu_v^{c_i-c_{i-1}}}{\longmapsto} & \\ &  &}&  \quad
    \xymatrix@R=10pt@C=20pt{   & \\& u_i & \\ &&  u_1 \ar[lu]_{c_{i}-c_1-1} \\  v  \ar[ruu]^{m} \ar[rru]_{m-c_{i}+c_1} \ar[rrd]_{c_j-c_i-1 } && \\ &&  u_j }  \quad \\
    && \\
    &&\\

    \mu_v^{c_1+1+(c_2-c_1)+\cdots +(c_{i-1}-c_{i-2})}(Q) && \mu_v^{c_1+1+(c_2-c_1)+\cdots +(c_{i}-c_{i-1})}(Q)

\end{array}\]

\vspace*{.5cm}

The same reasoning, explained with $u_1$ and $u_i$, applies to any pair of vertices $u_i$ and  $u_j$ with $i\neq j$. Assume $i<j$, then the new arrows created at step $i$ in the quiver $\mu_v^{c_1+1+(c_2-c_1)+\cdots +(c_{i}-c_{i-1})}(Q)$ are going to be deleted at step $j$ in the quiver $\mu_v^{c_1+1+\cdots + (c_{i}-c_{i-1})+ \cdots +(c_{j}-c_{j-1})}(Q)$. Similarly,  the arrows  which are canceled at step $i$  in the quiver $\mu_v^{c_1+1+(c_2-c_1)+\cdots +(c_{i}-c_{i-1})}(Q)$ are going to be recovered at step $j$ in the quiver $\mu_v^{c_1+1+\cdots + (c_{i}-c_{i-1})+ \cdots +(c_{j}-c_{j-1})}(Q)$.

Finally, at the last step we apply $\mu_v^{m-c_z}$ to the quiver  $\mu_v^{c_1+1+(c_2-c_1)+\cdots +(c_{z}-c_{z-1})}(Q)$. This step just changes the colour of the arrows leaving the vertex $v$.  The colour $c_k$ of the arrow $\xymatrix{v \ar[r]^{c_{k}} & u_k}$ of $Q$   changes to the colour $c_k-c_z-1$ after the first $z $ steps. Therefore, after applying $\mu_v^{m-c_z}$ it returns to the colour $c_k$ in $\mu_v^{c_1+1+(c_2-c_1)+\cdots +(c_z-c_{z-1})-(m-c_z)}(Q)=\mu_v^{m+1}(Q)$.

\end{proof}

We finish this section with two examples that show the importance of the previous lemma. Next example  shows that the converse of the above lemma is not true. In fact, the class of coloured quivers where the coloured mutation is invertible is wide bigger than the class $\mathcal{Q}_n^m$.

\begin{ejem}
 The following $2$-coloured quiver $Q$ does not belong to $\mathcal{Q}_3^2$ but $\mu_1^{3}(Q)=Q$.

%
%


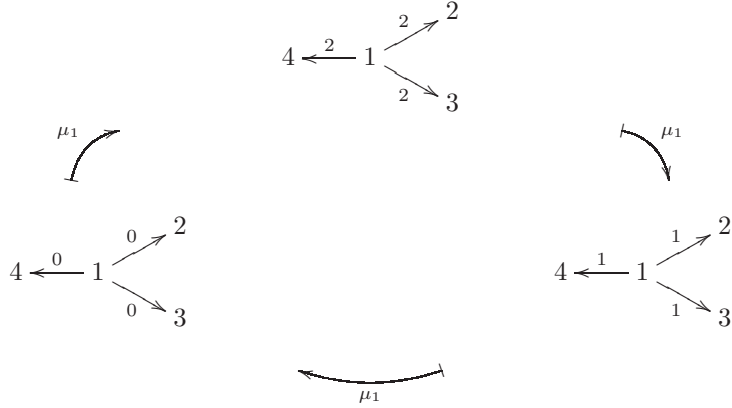
\[\begin{array}{ccccc}
& \quad \xymatrix@R=5pt@C=20pt{ & & 2 \\  4 &  1 \ar_{2}[l] \ar^{2}[ru] \ar_{2}[rd] &  \\ && 3 }\quad &\\
 \xymatrix@R=15pt@C=15pt{ & \\  \ar@{|->}@/^2mm/@<1.6ex>^{\mu_1}[ur] &  }  &  & \xymatrix@R=15pt@C=15pt{ \ar@{|->}@/^2mm/@<1.6ex>^{\mu_1}[dr] & \\  & }  \\
\quad  \xymatrix@R=5pt@C=20pt{  & & 2 \\  4 &  1 \ar_{0}[l] \ar^{0}[ru] \ar_{0}[rd] &  \\ && 3 }  \quad && \quad  \xymatrix@R=5pt@C=20pt{ & & 2 \\  4 &  1 \ar_{1}[l] \ar^{1}[ru] \ar_{1}[rd] &  \\ && 3 }\quad \\
&\xymatrix{   && \ar@{|->}@/^2mm/@<1.6ex>^{\mu_1}[ll] \\ && } &
\end{array}\]

\end{ejem}


The following example illustrate one of the most noticeable differences between the $m$-coloured mutation (for $m\geq 2$)  and the classical mutation (i.e., a $m$-coloured mutation with $m=1$) which is always an involution (i.e., $\mu^{2}= Id$).

\begin{ejem}\label{mu3 no id}
This example shows  a $2$-coloured quiver $Q$ such that $\mu_2^{i}(Q)$ belongs to $\mathcal{Q}_3^2$  for all $1\leq i$ but $Q$ does not. 

%

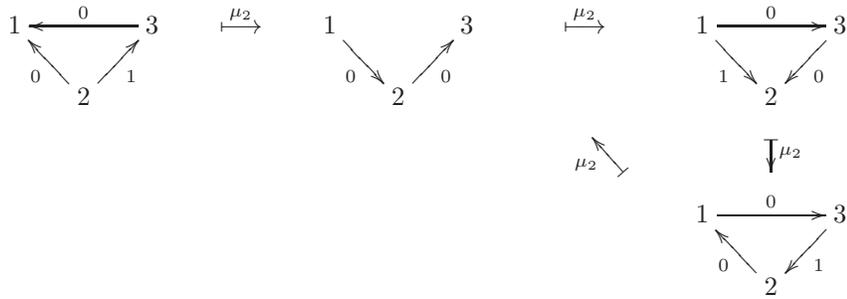
\[\begin{array}{ccccc}
      \quad \xymatrix@R=15pt@C=15pt{ 1 & & 3 \ar_{0}[ll] \\  &  2  \ar^{0}[lu] \ar_{1}[ru] &  } \quad
& \stackrel{\mu_2}{\longmapsto} & \quad \xymatrix@R=15pt@C=15pt{ 1 \ar_{0}[rd]& & 3  \\  &  2  \ar_{0}[ru] &  }\quad  &  \stackrel{\mu_2}{\longmapsto} & \quad  \xymatrix@R=15pt@C=15pt{ 1 \ar_{1}[rd]  \ar^{0}[rr]& & 3  \ar^{0}[ld] \\  &  2  &  }\quad  \\  &&& \xymatrix@R=12pt@C=10pt{  & & \\ & & \ar@{|->}^{\mu_2}[ul]}  & \xymatrix@R=12pt@C=10pt{  & \ar@{|->}^{\mu_2}[d] & \\ & & }  \\ &&&& \quad  \xymatrix@R=15pt@C=15pt{ 1 \ar^{0}[rr] & & 3 \ar^{1}[ld] \\  &  2  \ar^{0}[lu]}\quad
\end{array}\]

\vspace*{.3cm}

In particular, $\mu_2^{3}(Q)\neq Q$.

\end{ejem}

\vspace*{.3cm}

\section{Proof that $\mathcal{Q}_n^m$ is the coloured mutation class of $\mathbb{A}_n$}

We start this section defining two special types of cliques. Let $Q$ be a coloured quiver belonging to the class $\mathcal{Q}_n^m$ and $\mathcal{C}$  be a clique in $Q$. Let $\mathcal{A(C)}$ be the set of all arrows belonging to $\mathcal{C}$.

\begin{defi} The clique $\mathcal{C}$ is said to be:
\begin{enumerate}
 \item[$(a)$] \emph{almost extremal} if there is a connected component of  $Q\setminus \mathcal{A(C)}$ which
 is an  $\mathbb{A}_n$-quiver for $n\geq 1$.
 \item[$(b)$] \emph{extremal} if there is  a connected component of  $Q\setminus \mathcal{A(C)}$ which
 is an  $\mathbb{A}_1$-quiver (i.e., it consists of a single vertex).
\end{enumerate}

\end{defi}

Clearly, every extremal clique is an almost extremal clique.\\


\begin{ejem}

In the following quiver the clique determined by the vertices $2,3$ and $4$ is almost extremal but not extremal and the clique determined by the vertices $5,6,7$ and $8$ is extremal.
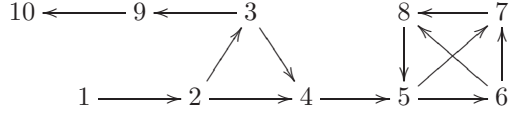
\begin{figure}[H]

\[
\xymatrix@R=20pt@C=10pt{ 10 && \ar[ll] 9&& \ar[rd] \ar[ll] 3 &&  & 8 \ar[d] && 7 \ar[ll]    \\  & 1 \ar[rr] && 2 \ar[ur] \ar[rr] && 4 \ar[rr] && 5 \ar[urr] \ar[rr] &&  6 \ar[u] \ar[ull] }
\]

\caption{Example of  a quiver with an extremal and an almost extremal clique.}
\end{figure}

\end{ejem}

We continue showing that any  quiver in $\mathcal{Q}_n^m$ has at least one almost extremal clique $\mathcal{C}$ and that every almost extremal clique $\mathcal{C}$ is mutation equivalent to an extremal clique.\\

\begin{lema}\label{existencia de ciclos casi extremales}
Let $Q$ be a coloured quiver belonging to the class $\mathcal{Q}_n^m$  which is not an $\mathbb{A}_n$-quiver. Then $Q$ has at least one almost extremal clique.

\end{lema}

\begin{proof}
Consider a path $\mathcal{P}=v_1\rightarrow v_2 \rightarrow \cdots \rightarrow v_m$ in $Q$ that does not repeat vertices and is the longest path in $Q$.
Since $Q$ is not an $\mathbb{A}_n$-quiver, $\mathcal{P}$ must pass through at least two vertices belonging to a clique $\mathcal{C}$  of size $r\geq 3$. Assume that $v_1 \in \mathcal{C}$. We claim that the connected component of  $Q\setminus \mathcal{A(C)}$ containing the vertex $v_1$ is an  $\mathbb{A}_1$-quiver and equivalently $\mathcal{C}$ is an extremal clique. Indeed, if it is not the case, there is a vertex $w \notin \mathcal{C} $ which is a neighbour of $v_1$. Since $Q$ has not holes we have that $w\notin \mathcal{P}$.  Therefore the path $w \rightarrow v_1\rightarrow v_2 \rightarrow \cdots \rightarrow v_m$ is a path in $Q$ longer than $\mathcal{P}$.

We now turn to the case $v_1 \notin \mathcal{C}$. Let $v_k$ be the first vertex of the path $\mathcal{P}$ with at least three neighbours  namely $v_{k-1}, v_{k+1}$ and $v_{k+2}$. Consider the clique $\mathcal{C}'$ to which the vertices $v_k $ and $v_{k+1}$ belong. Clearly $\mathcal{C}'$ is almost extremal since $Q\setminus \mathcal{A(C')}$ contains the component $v_1\rightarrow v_2 \rightarrow \cdots \rightarrow v_k$ of type $\mathbb{A}_k$.

\end{proof}

Given a quiver $Q$ with an almost extremal clique $\mathcal{C}$ next lemma allow us to choose the colouration for the connected component of  $Q\setminus \mathcal{A(C)}$ of type $\mathbb{A}_k$. This fact will be useful in the following.
\\


\begin{lema}
All colourations of $\mathbb{A}_n$ are mutation equivalent.
\end{lema}

\begin{proof}
It suffices to show that we can reach  every colouration of an $\mathbb{A}_n$-quiver by applying a sequence of mutation  to the following  particular colouration $ \xymatrix{n \ar^{0}[r] & n-1 \ar^{0}[r] &  \cdots  \ar^{0}[r] & 2  \ar^{0}[r] & 1  }$.

Clearly, for any colour $d_1 \in \{0,\cdots,m\}$, applying $\mu_1^{d_1}$  to the above quiver we obtain

$$  \xymatrix{n \ar^{0}[r] & n-1 \ar@{.}[r] & 3  \ar^{0}[r] & 2  \ar^{d_1}[r] & 1  }$$

\vspace*{.5cm}

Now, assume that we have coloured the first $k$ arrows  with any colours $d_1, d_2, \cdots, d_k \in \{0,\cdots,m\}$ and the remaining  arrows are still coloured with $0$. That is,

$$  \xymatrix{n \ar^{0}[r] & n-1  \ar@{.}[r]  & k+2 \ar^{0}[r] &  k+1 \ar^{d_k}[r]  & k  \ar^{d_{k-1}}[r] &  k-1 \ar@{.}[r] &  2 \ar^{d_{1}}[r] & 1  }$$

\vspace*{.5cm}

We want to show that we can  colour the first $k+1$ arrows  with  given colours $c_1, c_2, \cdots, c_{k+1} \in \{0,\cdots,m\}$ and leave the rest coloured with $0$. To this end,  consider $d_i=c_i$ for $ 1\leq i\leq k-2$,  $d_{k-1}=c_{k-1}+c_{k}+c_{k+1}-m$ and $d_{k}=m$ in the above coloured quiver. After applying the composition
$ \mu_k^{c_k+c_{k+1}-m}\circ \mu_{k+1}^{d_{k+1}}$  we obtain

$$  \xymatrix{n \ar^{0}[r] & n-1  \ar@{.}[r]  & k+2 \ar^{c_{k+1}}[r] &  k+1 \ar^{c_k}[r]  & k  \ar^{c_{k-1}}[r] &  k-1 \ar@{.}[r] &  2 \ar^{c_{1}}[r] & 1  }$$

We continue in this fashion obtaining any colouration $ \xymatrix{n \ar^{c_{n-1}}[r] & n-1 \ar^{c_{n-2}}[r] &  \cdots  \ar^{c_2}[r] & 2  \ar^{c_1}[r] & 1  }$.

%
%
%
%
%
%
%
%
%
%
%
%
\end{proof}

\begin{obs}
In fact, we have proved that the quiver $ \xymatrix{ *+[o][F]{Q'} \ar^{0}[r] & n-1 \ar^{0}[r] &  \cdots  \ar^{0}[r] & 2 \ar^{0}[r] & 1  }$ is mutation equivalent to the  quiver $ \xymatrix{ *+[o][F]{Q'} \ar^{c_{n-1}}[r] & n-1 \ar^{c_{n-2}}[r] &  \cdots  \ar^{c_2}[r] & 2 \ar^{c_1}[r] & 1  }$  where the decoration $\SelectTips{eu}{10}\xymatrix@R=.2pc@C=.8pc{ *+[o][F]{Q'}}$ means that $Q'$ can be any subquiver.

\end{obs}

Let $Q$ be a coloured quiver belonging to the class $\mathcal{Q}_n^m$  containing an almost extremal clique $\mathcal{C}$ with a connected component of  $Q\setminus \mathcal{A(C)}$ of type $\mathbb{A}_2$. The following lemma shows that  $Q$ is mutation equivalent to a quiver $Q'$ where $\mathcal{C}$ was moved  by transforming into  an extremal clique.

\begin{lema}\label{transformar un ciclo casi extremal en uno extremal}

Let $\mathcal{C}$ be an almost extremal clique with $n+1$ vertices $v,v_1,\cdots, v_{n}$.
Assume that $v$ has exactly $n+1$ neighbours: the vertices $v_1, \cdots, v_{n}$ and another vertex $w_1$; in such a way that the vertex $v$ together with the vertex $w_1$  form a clique with $2$ vertices. Moreover, the vertex $w_1$ does not have any other neighbour. We illustrate the situation in the following quiver $Q$:

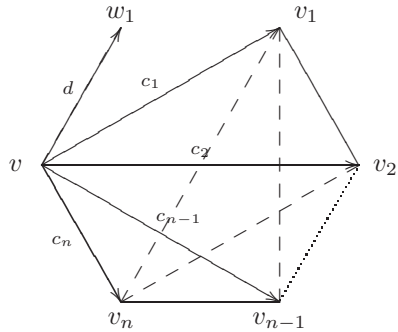
\[
\xy/r5pc/:{\xypolygon6"A"{~<{}~>{}{}}}
\POS"A6" \drop{ \begin{array}{c} \\  v_{n-1}\end{array}}
\POS"A2" \ar@{-} "A1"
\POS"A4" \ar@{->}  ^{d}   "A3"
\POS"A4" \ar@{->}  ^{c_1}   "A2"
\POS"A4" \ar@{->}  ^{c_2}   "A1"
\POS"A4" \ar@{->}  ^{c_{n-1}}   "A6"
\POS"A4" \ar@{->}  _{c_{n}}   "A5"
\POS"A4" \ar@{-}   "A5"
\POS"A5" \ar@{-}   "A6"
\POS"A6" \ar@{.}   "A1"
\POS"A5" \drop{ \begin{array}{c} \\  v_{n}\end{array}}
\POS"A4" \drop{\begin{array}{ccc} v &&  \end{array}}
\POS"A1" \drop{ \begin{array}{ccc} && v_{2}\end{array}}
\POS"A2" \drop{ \begin{array}{ccc} && v_{1}\\ && \end{array}}
\POS"A3" \drop{ \begin{array}{ccc}&w_{1}& \\ && \end{array}}

\POS"A2" \ar@{--} "A5"
\POS"A2" \ar@{--} "A6"
\POS"A3" \ar@{--} "A4"
\POS"A1" \ar@{--} "A5"
\endxy   \]

Then, $Q$ is mutation equivalent to the  quiver:

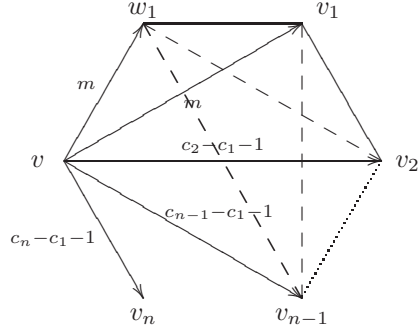
\[
\xy/r5pc/:{\xypolygon6"A"{~<{}~>{}{}}}
\POS"A6" \drop{ \begin{array}{c} \\  v_{n-1}\end{array}}
\POS"A2" \ar@{-} "A1"
\POS"A3" \ar@{-}   "A2"
\POS"A4" \ar@{->}  ^{m}   "A3"
\POS"A4" \ar@{->}  _{ m}   "A2"
\POS"A4" \ar@{->}  ^{c_2-c_1-1}   "A1"
\POS"A4" \ar@{->}  ^{c_{n-1}-c_1-1}   "A6"
\POS"A4" \ar@{->}  _{c_{n}-c_1-1}   "A5"
\POS"A6" \ar@{.}   "A1"
\POS"A5" \drop{ \begin{array}{c} \\  v_{n}\end{array}}
\POS"A4" \drop{\begin{array}{ccc} v &&  \end{array}}
\POS"A1" \drop{ \begin{array}{ccc} && v_{2}\end{array}}
\POS"A2" \drop{ \begin{array}{ccc} && v_{1}\\ && \end{array}}
\POS"A3" \drop{ \begin{array}{ccc}&w_{1}& \\ && \end{array}}

\POS"A3" \ar@{--} "A1"
\POS"A2" \ar@{--} "A6"
\POS"A3" \ar@{--} "A6"
\POS"A3" \ar@{--} "A6"
\endxy   \]

where  $w_1$ has exactly $n$ neighbours and the clique formed by the vertices  $v,w_1,v_1, \cdots,v_{n-1}$ is an extremal clique of size $n+1$.

\end{lema}

\begin{proof}
Observe that the vertices $v_1, \cdots, v_{n}$ can possibly belong each of them also to at most another clique in such a way that condition 1 of the definition of $\mathcal{Q}_n^m$ is satisfied.

Label $c_i$ the colour of the arrow $\xymatrix{v \ar[r]^{c_{i}} & v_i}$ and $d$ the colour of the arrow $\xymatrix{v \ar[r]^{d} & w_1}$.
By Lemma \ref{triangulos admisibles}, we know that all the colours $c_i$ are distinct. Hence we can assume that $c_1<c_2<\cdots < c_{n}$.

We can certainly assume that $c_1=d$. Otherwise, we apply  $\mu_{w_1}^{m-d+c_1+1}$. Then, after applying $\mu_{v}^{c_1}$ we obtain:

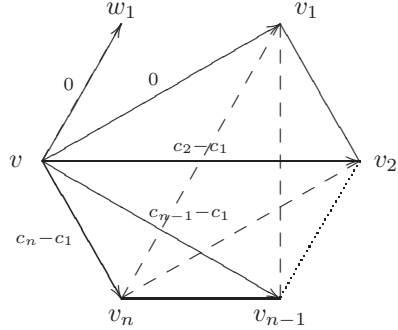
\[
\xy/r5pc/:{\xypolygon6"A"{~<{}~>{}{}}}
\POS"A6" \drop{ \begin{array}{c} \\  v_{n-1}\end{array}}
\POS"A2" \ar@{-} "A1"
\POS"A4" \ar@{->}  ^{0}   "A3"
\POS"A4" \ar@{->}  ^{0}   "A2"
\POS"A4" \ar@{->}  ^{c_2-c_1}   "A1"
\POS"A4" \ar@{->}  ^{c_{n-1}-c_1}   "A6"
\POS"A4" \ar@{->}  _{c_{n}-c_1}   "A5"
\POS"A4" \ar@{-}   "A5"
\POS"A5" \ar@{-}   "A6"
\POS"A6" \ar@{.}   "A1"
\POS"A5" \drop{ \begin{array}{c} \\  v_{n}\end{array}}
\POS"A4" \drop{\begin{array}{ccc} v &&  \end{array}}
\POS"A1" \drop{ \begin{array}{ccc} && v_{2}\end{array}}
\POS"A2" \drop{ \begin{array}{ccc} && v_{1}\\ && \end{array}}
\POS"A3" \drop{ \begin{array}{ccc}&w_{1}& \\ && \end{array}}

\POS"A2" \ar@{--} "A5"
\POS"A2" \ar@{--} "A6"
\POS"A3" \ar@{--} "A4"
\POS"A1" \ar@{--} "A5"
\endxy   \]

and after applying $\mu_v$ again we get the desired quiver. Observe that Lemma \ref{mutacion invertible} allows us to reverse this procedure and show that the quiver $Q$ can be obtained from $\mu_v^{c_1+1}(Q)$ by a sequence of mutations.

\end{proof}

\begin{obs}

Let $Q$ be a coloured quiver belonging to the class $\mathcal{Q}_n^m$  containing an almost extremal clique $\mathcal{C}$ with a connected component of  $Q\setminus \mathcal{A(C)}$ of type $\mathbb{A}_k$ for  $k\geq 3$. The same proof of the above lemma shows that  $Q$ is mutation equivalent to a quiver $Q'$ where $\mathcal{C}$ was "moved"  by transforming into another almost extremal clique $\mathcal{C'}$ with a connected component of  $Q'\setminus \mathcal{A(C')}$ of type $\mathbb{A}_{k-1}$.

\end{obs}

We ilustrate the case $k=3$: the quiver $ \xymatrix{ w_2 &   w_1 \ar^{}[l] & *+[o][F]{\mathcal{C}} \ar^{}[l]}$ is mutation equivalent to the  quiver $\xymatrix{ w_2 & *+[o][F]{\mathcal{C}'} \ar^{}[l] & v_n \ar^{}[l]}$.\\

The following lemma shows that we can reduce the size of an extremal clique using coloured mutations or  more precisely;

\begin{lema}\label{transformar un ciclo extremal de tamaño n en uno de tamaño n-1}
Assume that $v$ is a vertex with exactly $n-1$ neighbours $v_1, \cdots, v_{n-1}$, all of them belonging to the same clique $\mathcal{C}_n$ as in the following diagram:

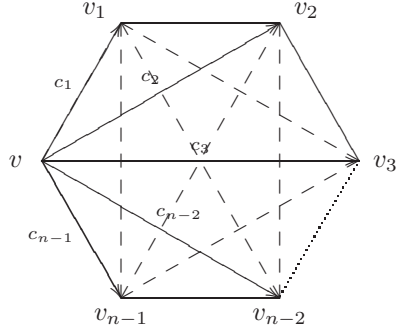
\[
\xy/r5pc/:{\xypolygon6"A"{~<{}~>{}{}}}

\POS"A6" \drop{ \begin{array}{c} \\  v_{n-2}\end{array}}
\POS"A2" \ar@{-} "A1"
\POS"A3" \ar@{-}   "A2"
\POS"A4" \ar@{->}  ^{c_1}   "A3"
\POS"A4" \ar@{->}  ^{c_2}   "A2"
\POS"A4" \ar@{->}  ^{c_3}   "A1"
\POS"A4" \ar@{->}  ^{c_{n-2}}   "A6"
\POS"A4" \ar@{->}  _{c_{n-1}}   "A5"
\POS"A4" \ar@{-}   "A5"
\POS"A5" \ar@{-}   "A6"
\POS"A6" \ar@{.}   "A1"
\POS"A5" \drop{ \begin{array}{c} \\  v_{n-1}\end{array}}
\POS"A4" \drop{\begin{array}{ccc} v &&  \end{array}}
\POS"A1" \drop{ \begin{array}{ccc} && v_{3}\end{array}}
\POS"A2" \drop{ \begin{array}{ccc} && v_{2}\\ && \end{array}}
\POS"A3" \drop{ \begin{array}{ccc} v_{1}&& \\ && \end{array}}

\POS"A2" \ar@{--} "A5"
\POS"A2" \ar@{--} "A6"
\POS"A3" \ar@{--} "A4"
\POS"A3" \ar@{--} "A5"
\POS"A3" \ar@{--} "A6"
\POS"A3" \ar@{--} "A1"
\POS"A1" \ar@{--} "A5"

\endxy   \]

Then, $\mathcal{C}_n$ is mutation equivalent to the  quiver:

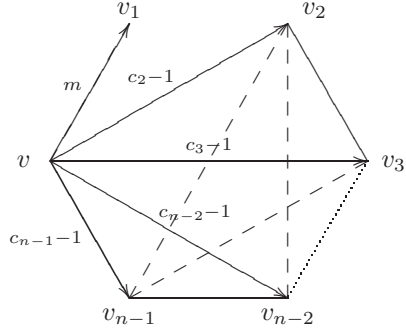
\[
\xy/r5pc/:{\xypolygon6"A"{~<{}~>{}{}}}

\POS"A6" \drop{ \begin{array}{c} \\  v_{n-2}\end{array}}
\POS"A2" \ar@{-} "A1"
\POS"A4" \ar@{->}  ^{m}   "A3"
\POS"A4" \ar@{->}  ^{c_2-1}   "A2"
\POS"A4" \ar@{->}  ^{c_3-1}   "A1"
\POS"A4" \ar@{->}  ^{c_{n-2}-1}   "A6"
\POS"A4" \ar@{->}  _{c_{n-1}-1}   "A5"
\POS"A4" \ar@{-}   "A5"
\POS"A5" \ar@{-}   "A6"
\POS"A6" \ar@{.}   "A1"
\POS"A5" \drop{ \begin{array}{c} \\  v_{n-1}\end{array}}
\POS"A4" \drop{\begin{array}{ccc} v &&  \end{array}}
\POS"A1" \drop{ \begin{array}{ccc} && v_{3}\end{array}}
\POS"A2" \drop{ \begin{array}{ccc} && v_{2}\\ && \end{array}}
\POS"A3" \drop{ \begin{array}{ccc}&v_{1}& \\ && \end{array}}

\POS"A2" \ar@{--} "A5"
\POS"A2" \ar@{--} "A6"
\POS"A3" \ar@{--} "A4"
\POS"A1" \ar@{--} "A5"
\endxy   \]

 where the vertex $v$ together with the vertices $v_2,\cdots, v_{n-1}$ form a clique with $n-1$ vertices and together with the vertex $v_1$ a clique with $2$ vertices.
\end{lema}

\begin{proof}
Observe that the vertices $v_1, \cdots, v_{n-1}$ can possibly belong each of them also to at most another clique in such a way that condition 1 of the definition of $\mathcal{Q}_n^m$ is satisfied.

Without loss of generality we can assume that $c_1=0$. Otherwise, we apply $\mu_v^{c_1}$.  Then, the claim follows after applying $\mu_v$. Observe that Lemma \ref{mutacion invertible} allows us to reverse this procedure.

\end{proof}

\vspace*{.5cm}

Now, we are ready to proof the main result of this paper.\\

\begin{teo}
An $m$-coloured quiver $Q$ is mutation equivalent to $\mathbb{A}_n$ if and only if it is in  $\mathcal{Q}_n^m$.
\end{teo}

\begin{proof}

Obviously, all colourations of an  $\mathbb{A}_n$-quiver are in $\mathcal{Q}_n^m$. Let $Q$ be a coloured quiver belonging to $\mathcal{Q}_n^m$  which is not an  $\mathbb{A}_n$-quiver. We claim that $Q$ is mutation equivalent to an  $\mathbb{A}_n$-quiver.  It follows upon executing the following: \newline

{\bf Algorithm. } \begin{description}
 \item[Step 1] Choose an almost extremal clique  $\mathcal{C}$ in $Q$.
 \item[Step 2] Use the processes described in lemmata \ref{transformar un ciclo casi extremal en uno extremal} and \ref{transformar un ciclo extremal de tamaño n en uno de tamaño n-1} to obtain a new coloured quiver $Q'$ which is mutation equivalent to $Q$. In $Q'$ the clique $\mathcal{C}$  was transformed into a new clique $\mathcal{C}'$ with exactly one vertex less.
 \item[Step 3] If $Q'$ is an $\mathbb{A}_n$-quiver we are done. If not, go to step 1 with $Q'$ playing the r\^ole of $Q$.
\end{description}

This process must stop after a finite number of steps, since the number of cliques; and their sizes,  are finite.\\

Lemma \ref{mutacion invertible} allows us to reverse this procedure and show that every quiver of $\mathcal{Q}_n^m$ can be reached by iterated mutation on an $\mathbb{A}_n$-quiver.  The proof is completed by noting that, by Lemma \ref{cerrado por mutaciones}, the class $\mathcal{Q}_n^m$ is closed under coloured mutation.
\end{proof}

\section{Weight induced by the colouration and clique number of a coloured quiver }


Throughout this section we assume $Q$ be a simple $m$-coloured quiver with  colouration function $\kappa:Q_1 \rightarrow \{0,1,\ldots, m\}$. Denote by $\mathcal{P}$ the set of all non-trivial paths in $Q$. We start this section defining  a weight function in the set $\mathcal{P}$ which is induced by the colouration $\kappa$.\\

\begin{defi}
The weight function $w:\mathcal{P}\rightarrow \mathbb{N}$ induced by the colouration $\kappa$ is defined in the following way: if $p=x_1\cdots x_k x_{k+1}$ is a non-trivial path in $Q$, then $$w(p):= \sum_{i=1}^{k}\kappa(x_{i}x_{i+1}).$$
\end{defi}

We also denote the weight of the path $x_1\cdots x_k x_{k+1}$ by $\overline{x_1\cdots x_k x_{k+1}}$. For example, the weight of the $3$-cycle $(xyz)$ is $\overline{(xyz)}= \overline{xy} + \overline{yz} + \overline{zx}=c_{xy}+c_{yz}+c_{zx}$.

 Next, we introduce a new quantity which is useful to our propose.\\

In the following  assume $k$ to be at least equal to $3$.

\begin{defi}
Let  $K=\{x_1,\ldots, x_k\}$  be a $k$-clique in $Q$. The energy of the clique $K$ is defined as the minimum possible weight of a $k$-cycle contained in $K$ ; i.e. 
$$  \delta(K):= \min\{ \overline{(x_{\theta(1)}x_{\theta(2)}\cdots x_{\theta(k)})}: \theta \in S_k  \}  $$ where  $S_k$ is the symmetric group over the set $\{1,2,\ldots,k\}$.
\end{defi}

Clearly $\delta(K)$ is greater or equal to $0$ and the equality $\delta(K)=0$ holds when there is an oriented cycle in $K$ with all its arrows coloured with $0$.  Corollary \ref{ciclo color cero} implies that the cliques  $\mathcal{C}$  in $\mathcal{Q}_n^m$ with exactly $m+2$ vertices have $\delta(\mathcal{C})=0$. We want to see  that if a $k$-clique $K$  has $\delta(K)=0$, then $k$ have to be exactly $m+2$.\\

 Remember that the Minkowski sum of $A,B \subseteq \mathbb{Z}$ is $A+B := \{a+b: a\in A, b\in B\}$. In particular, the $n$-fold Minkowski sum of a set $A\subseteq \mathbb{Z}$ with itself is denoted by $nA=A+A+\cdots+A$ ($n$ terms). If $A\subseteq \mathbb{Z}$ and $h\in \mathbb{Z}$, the translate of $A$ by $h$ is $A+h:= \{a+h: a\in A\}$.\\

It is proved in Remark \ref{Remark color triangulo} that every  triangle $(xyz)$ in some quiver in the mutation class of $\mathbb{A}_{n}$ satisfies $\overline{xy}+\overline{yz}+\overline{zx} \in \{m-1,2m+1\}$. This fact is equivalent to the following proposition:\\

\begin{prop}\label{prop-k3}
If $Q$ is an $m$-coloured quiver in the mutation class of $\mathbb{A}_{n}$ and $K$ is a $3$-clique in $Q$ then $\delta(K)=m-1$.
\end{prop}

We need the following lemma which compute the possible weight values of a $k$-cycle.

\begin{lema}\label{lemma-gaps}
Let $Q$ be an $m$-coloured quiver in the mutation class of $\mathbb{A}_{n}$ and $K=\{x_1,\ldots, x_k\}$ be a $k$-clique in $Q$. Let $A=\{m-1,2m+1\}$. Then, $\overline{(x_1\cdots x_k)} \in (k-2)A-(k-3)m .$ 

\end{lema}

\begin{proof}
Since $Q$ is skew-symmetric,  $\overline{x_{i+1}x_{1}}+\overline{x_1x_{i+1}}=m$ for all $i : 2\leq i \leq n-2$. Then, the $k$-cycle $(x_1\cdots x_k)$ has weight $\overline{(x_1\cdots x_k)} = \sum_{i=2}^{k-2} \overline{(x_1x_{i}x_{i+1})} - (k-3)m \in (k-2)A-(k-3)m$.

\end{proof}

\begin{obs}
We have $(k-2)A-(k-3)m=\{(k-2-i)(m-1)+i(2m+1)-(k-3)m: 0\leq i \leq k-2\}$. With $i=0$ we obtain $(k-2)(m-1)-(k-3)m = m+2-k$, the minimum possible value of $(k-2)A-(k-3)m$, and with $i=1$ we have $(k-3)(m-1)+(2m+1)-(k-3)m = 2m+4-k$, the second minimum possible value.

\end{obs}

In particular, we have the following corollary. 

\begin{coro}
Let $Q$ be an $m$-coloured quiver in the mutation class of $\mathbb{A}_{n}$ and $K=\{x_1,\ldots, x_k\}$ be a $k$-clique in $Q$ such that $\overline{(x_1\cdots x_k)}<2m+4-k$, then  $\overline{(x_1\cdots x_k)}=m+2-k$.
\end{coro}

Proposition \ref{prop-k3} can be generalized in the following way.\\

\begin{teo}\label{theo-m2}
If $Q$ is an $m$-coloured quiver in the mutation class of $\mathbb{A}_{n}$ and $K$ is a $k$-clique in $Q$ then $\delta(K)=m+2-k$.
\end{teo}

\begin{proof}
The result for $k=3$ follows from Proposition \ref{prop-k3}. Now assume that the result is true for some value of $k\geq 3$ and consider a $(k+1)$-clique $K=\{x_1,x_2,\ldots,x_{k+1}\}$ in $Q$. Then, reordering indices if necessary, we can suppose  that $\overline{(x_1x_2\cdots x_k)}=m+2-k$. Now we consider the cycles $c_1=(x_1 x_{k+1}x_2 x_3 \cdots x_k)$, $c_2=(x_1x_2 x_{k+1} x_3 \cdots x_k)$, $\ldots$ , $c_{k}=(x_1x_2 x_3 \cdots x_k x_{k+1})$ and calculate the sum weight of these cycles.
\begin{eqnarray*}
 \sum_{i=1}^{k} \overline{c_i} &=&  \sum_{i=1}^{k} \overline{x_ix_{k+1}} +  \sum_{i=1}^{k} \overline{x_{k+1}x_{i}} + (k-1)(\overline{x_1x_2} + \overline{x_2x_3} + \cdots + \overline{x_{k-1}x_{k}} + \overline{x_{k}x_1} ) \\
                               &=&  mk + (k-1)(m+2-k) \\
                               &=& (2m+3-k)\cdot k - (m+2)
\end{eqnarray*}
Then, the average weight of those cycles is $\frac{1}{k}\sum_{i=1}^{k} \overline{c_i}= 2m+3-k - \frac{m+2}{k}< 2m+3-k$. This implies the existence of some cycle $c_j$ with $1\leq j \leq k$ such that $\overline{c_j}<2m+3-k = 2m+4-(k+1)$. By Lemma \ref{lemma-gaps}, the only possibility is $\overline{c_j}=m+2-(k+1)$ which is the minimum possible weight of a $(k+1)$-cycle. Therefore  the energy of the $(k+1)$-clique $K$ is $\delta(K)=m+2-(k+1)$.


\end{proof}

In particular, we can find an upper bound for the  clique number  $\omega(Q)$ of a quiver $Q$ in the class $\mathcal{Q}_n^m$. Moreover, if $Q'$  is a quiver in the class $\mathcal{Q}_n^m$ with a $(m+2)$-clique $K$, then $\omega(Q')=m+2$. 

\begin{coro}
Let $Q$ be an $m$-coloured quiver in $\mathcal{Q}_n^m$ with clique number $\omega(Q)$. Then, $\omega(Q)\leq m+2$ and this upper bound is reached by a quiver in that class.
\end{coro}


\section{The $0$-coloured part of a quiver in the mutation class of $\mathbb{A}_n$}

The $0$-coloured part of a quiver in the mutation class of $\mathbb{A}_n$ plays an important role in the study of the $m$-cluster tilted algebra of type $\mathbb{A}_n$. In fact, every quiver of an $m$-cluster tilted algebra of type $\mathbb{A}_n$  can be obtained  by considering the $0$-coloured part of the  quiver in the class $\mathcal{Q}_n^m$. In particular, the lengths of the cycles in the $0$-coloured part bring information about the possible cycles that an $m$-cluster tilted algebra of type $\mathbb{A}_n$ can have.  In \cite{Murphy}, using different tools,  it is proved that every cycle in the $0$-coloured part (if there is any) has length $m+2$. Next we bring a completely combinatorial proof of this result, based on Theorem \ref{theo-m2}.\\

Another important consequence of the above Theorem is the following. If $Q'$ is the $0$-coloured part of an $m$-coloured quiver $Q$ in the mutation class of $\mathbb{A}_{n}$, then $Q'$ admits only cycles of length $m+2$.\\

\begin{coro}\label{los unicos ciclos} Let $Q$ be an $m$-coloured quiver in the mutation class of $\mathbb{A}_{n}$ and consider its $0$-coloured part $Q'$. Suppose there is a $k$-cycle $c=(x_1x_2\cdots x_k)$ in $Q'$. Then, $k=m+2$.
\end{coro}

\begin{proof}
By our characterization of quivers in the mutation class of $\mathbb{A}_{n}$ we have that the set $K=\{x_1,x_2,\ldots, x_k\}$ is a $k$-clique in $Q$ because $Q$ has not holes. This clique contains the $k$-cycle $c$ which has zero weight (because it is in $Q'$), then $\delta(K)=0$. By Theorem \ref{theo-m2}, $\delta(K)=m+2-k$ and we conclude $k=m+2$.
\end{proof}

We finish this section with a proposition which characterizes  the quivers of an $m$-cluster tilted algebra of type $\mathbb{A}_n$. This proposition has  already been proved in  \cite{Murphy}. However, we can now give an independent argument.

\begin{prop}
Let $Q$ be the quiver of an $m$-cluster tilted algebra of type $\mathbb{A}_n$. Then,

\begin{enumerate}
  \item For every vertex $x\in Q_0$ the sets $s^{-1}(x)$ and $t^{-1}(x)$ have cardinality at most two;
  \item the only possible cycles which can occur in $Q$  are oriented cycles of length $m+2$.
\end{enumerate}

\end{prop}

\begin{proof}

We have to  consider the $0$-coloured part of the quivers in the class $\mathcal{Q}_n^m$. Then, Lemma \ref
{triangulos admisibles} implies $(1)$  and $(2)$ follows directly from Corollary \ref{los unicos ciclos}.
\end{proof}

\medskip

\begin{ejem} If we compute the $0$-coloured part of the $2$-coloured quiver  of Figure \ref{Q132}, we obtain a non connected quiver of a $2$-cluster tilted algebra of type  $\mathbb{A}_{13}$.

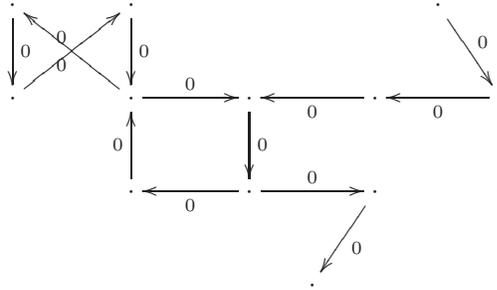
\begin{figure}[H]
\begin{center}
\[
\xymatrix@R=25pt@C=15pt{ \cdot \ar[d]^0 && \cdot \ar[d]^0 && &&  & \cdot \ar[dr]^0 & \\
\cdot \ar[rru]^{0} && \cdot  \ar[rr]^0 \ar[llu]^{ 0 }   && \cdot \ar[d]^0 && \cdot \ar[ll]^0 && \cdot \ar[ll]^0 \\
&& \cdot \ar[u]^0  && \cdot  \ar[ll]^0 \ar[rr]^0 && \cdot  \ar[dl]^0 && \cdot \\
&& && & \cdot  & &&
}
\]

\caption{The $0$-coloured part of the $2$-coloured quiver  of figure \ref{Q132} }
\end{center}
\end{figure}

\end{ejem}

\end{document}